\documentclass[11pt]{amsart}
\usepackage{amssymb}

\newtheorem{proposition}{Proposition}[section]
\newtheorem{lemma}[proposition]{Lemma}
\newtheorem{corollary}[proposition]{Corollary}
\newtheorem{theorem}[proposition]{Theorem}
\newtheorem{remark}[proposition]{Remark}

\theoremstyle{definition}

\newcommand{\selabel}[1]{\label{se:#1}}

\def\<{\leq}
\def\>{\geq}

\def\e{\varepsilon}

\def\ot{\otimes}

\date{}

\begin{document}
\title{The Green rings of the generalized Taft Hopf algebras}
\author{Libin Li}
\address{School of Mathematical Science, Yangzhou University,
Yangzhou 225002, China}
\email{lbli@yzu.edu.cn}
\author{Yinhuo Zhang}
\address{Department WNI, University of Hasselt, Universitaire Campus, 3590 Diepeenbeek,Belgium }
\email{yinhuo.zhang@uhasselt.be}
\begin{abstract}
In this paper, we investigate the Green ring $r(H_{n,d})$ of the
generalized Taft algebra $H_{n,d}$, extending the results of Chen,
Van Oystaeyen and Zhang in \cite{Coz}. We shall determine all
nilpotent elements of the Green ring $r(H_{n,d})$. It turns out that
each nilpotent element in  $r(H_{n,d})$ can be written as a sum of
indecomposable projective representations. The Jacobson radical
$J(r(H_{n,d}))$ of $r(H_{n,d})$ is generated by one element,
 and its rank is $n-n/d$. Moreover, we will present all the finite
dimensional indecomposable representations over the complexified
Green ring $R(H_{n,d})$ of $H_{n,d}.$  Our analysis is based on the
decomposition of the tensor product of indecomposable
representations and the observation of the solutions for the system of  equations associated to the generating relations of the Green ring $r(H_{n,d})$.
\end{abstract}

\maketitle
\section{\bf Introduction}

Let $\mathcal{C}$ be the category of finite dimensional
representations of a Hopf algebra $H$ over a field $K$. In the study
of the monoidal structure of $\mathcal{C}$ one has to consider the
decomposition of  the tensor  product of  representations in
$\mathcal{C},$ in particular, the tensor product of two
indecomposable representations in $\mathcal{C}.$  However, in
general, very little is known about how a tensor product of
two indecomposable representations decomposes into a direct sum of
indecomposable representations. One method of addressing this problem is
to consider the tensor product as the multiplication of the  Green ring (or
the representation ring) $r(H)$, and  to study the
ring properties of the Green ring. In general,  it is relatively easy  to explore the ring structure of the complexified Green ring $R(H)$ as we shall see in this paper.

A lot of work have been done in this direction. Firstly, Green \cite{Green, GR}, Benson and Carlson, etc., considered the semi-simplicity  of the
representation ring $r(KG)$ for modular representations of a finite
group $G$ (see \cite{B}). One of the interesting results they obtained is  that $KG$ is of finite representation type if and only if there are no nilpotent elements in $r(KG)$. In general, it is difficult to determine all nilpotent elements of $r(KG)$ if $KG$ is of infinite representation type (see \cite{Ar, He, Re}). For the Green rings of Hopf algebras,
if $H$ is a finite dimensional semi-simple Hopf algebra, then the Green ring $r(H)$ is equal to the Grothendieck ring and is semi-simple (see, e.g. \cite{Lo,With}). If $H$ is the enveloping algebra of a complex semi-simple Lie algebra, the Green ring has been studied by Cartan and Weyl (see \cite{FH}). Here We would like to mention the recent work by Sergeev and Veselov
for basic classical Lie superalgebras (see \cite{SV}).
Darp\"o and Herschend have presented a general description of the Green ring of the polynomial Hopf algebra $K[x]$ in \cite{DH} in case the ground field $K$ is perfect. For representation rings of quantum algebras, we refer to the work by Domokos and Lenagan (see \cite{DL}) and the work by Chin (see \cite{Chin}).

 Recently,  Cibils in \cite{Cib} determined all the graded Hopf algebras on a cycle path coalgebra (which are just equal to the generalized Taft algebras (see \cite{R,Ta})),  and considered the decomposition of two indecomposable representations (see also \cite{Gunn}). Moreover,
 Cibils also computed the Green ring of the Sweedler $4$-dimensional Hopf algebra by generators and relations  and asked how to compute  the Green ring of $K\mathbb{Z}_n(q)/I_d$. In \cite{W},  Wakui computed the Green
 rings of all non semisimple Hopf algebras of dimension 8, one of which is $K\mathbb{Z}_4(q)/I_2$, over an algebraically closed field $K$ of characteristic 0
 by generators and relations.
More recently, Chen, Van Oystaeyen and Zhang (see \cite{Coz}) has explicitly
described the Green ring $r(H_n)$ of the Taft Hopf algebra
$H_n$ by two generators and two relations.

The aforementioned works motivate us to investigate the structure of the Green ring $r(H_{n,d})$ of the generalized Taft Hopf algebra $H_{n,d},$
which in turn  could be helpful to understand the structure and the classification of non-semisimple monoidal categories. It is well-known that
the pointed Hopf algebra $H_{n,d}$ is neither commutative nor cocommutative, and is even not quasitriangular in general.  However, it is of finite
representation type, and monomial,  that is, all the relations are given by paths.  Thus the study of the Green ring of $H_{n,d}$ can help us to understand more the tensor categories of finite representation type.

The main aim of this paper is to compute the generators and generating relations of  the Green ring $r(H_{n,d})$, to
determine explicitly all nilpotent elements of  $r(H_{n,d})$, and to give all finite dimensional indecomposable representations of the
complexified Green ring  $R(H_{n,d})$.  The paper is organized as follows.  In Section 2, we recall the  definitions and the basic representation theory of  the generalized Taft algebra $H_{n,d}$.  There are $nd$ non-isomorphic finite dimensional indecomposable modules over $H_{n,d}.$  For each $1\leqslant l\leqslant d$, there are exactly $n$ indecomposable
$H_{n,d}$-modules with dimension $l$. Every indecomposable
projective $H_{n,d}$-module is $d$-dimensional. In section 3, we
prove that  the Green ring $r(H_{n,d}) $ of $H_{n,d}$ is generated by two elements subject to two relations, one of which is associated with a generalized Fibonacci polynomial (see Theorem 3.2).

In order to describe all nilpotent elements and the structure of the Green rings $r(H_{n,d})$ and $R(H_{n,d})$, we compute, in Section 4,
explicitly the roots of the generalized Fibonacci polynomials appeared in the generating relations of the Green rings $r(H_{n,d})$.
It turns out that the generalized Fibonacci polynomial $F_n(a, x)$ has $n-1$ distinct complex roots (see Proposition 4.2 for detail).

In section 5, we analyze the distinct  solutions of the system of
equations associated to the aforementioned generalized Fibonacci
polynomial. By using a one to one correspondence between the set of
solutions for the system of equations and the set of the isomorphism
classes of irreducible modules over $R(H_{n,d})$,  we obtain
$nd-n+\frac{n}{d}$ irreducible modules of dimension 1 over
$R(H_{n,d})$. This gives us a description of all nilpotent elements
of $r(H_{n,d})$, and concludes that the Jacobson radical
$J(r(H_{n,d}))$ is a principal ideal (i.e. generated by one
element). Moreover, we will determine all finite dimensional
reducible and indecomposable modules over the Green ring $H_{n,d}$.
It is interesting that the number of indecomposable representations
of the Green ring $R(H_{n,d})$ and that of $H_{n,d}$ are both equal
to $nd$. In the final section, we apply the same technique used on
$r(H_{n,d})$ to describe all nilpotent elements and finite
dimensional indecomposable representations of the projective class
ring of $H_{n,d}$. Cibils studied the projective class rings in
\cite{CC} for basic and split Hopf algebras.

Throughout, we work over a fixed algebraically closed field $K$ of
characteristic 0. Unless otherwise stated, all algebras, Hopf
algebras and modules are defined over $K$; all modules are left
modules and finite dimensional; all maps are $K$-linear; dim and
$\otimes$ stand for $\mbox{\rm dim}_K$ and $\otimes_K$,
respectively. For the theory of Hopf algebras and quantum groups, we
refer to \cite{Ka, Maj, Mon, Sw}.

{\section{\bf Representation theory of the generalized Taft algebra $H_{n,d}$}\selabel{2}}

In the sequel, we fix  two integers $n,d\geqslant 2$ such that $d\mid n$. Assume that $q$ is a primitive $d$-th root of unity.  In
\cite{R} (see also \cite{AS}) Radford considered the following Hopf
algebra $H_{n,d}=H_{n,d}(q)$  generated by two elements $g$ and $h$
subject to the relations:
$$g^n=1,\ h^d=0,\ hg=qgh.$$
The algebra $H_{n,d}$ is a Hopf algebra with comultiplication
$\Delta$, counit $\e$, and antipode $S$ given by
$$\begin{array}{lll}
\Delta(g)=g\ot g,\ \   &
\Delta(h)=1\ot h+h\ot g,&\e(g)=1,\\
\e(h)=0,&S(g)=g^{-1}=g^{n-1},& S(h)=-q^{-1}g^{n-1}h.\\
\end{array}$$
Note that the dimension of $H_{n,d}$ is $dn$,  and the set $\{g^ih^j|0\leqslant i\leq n-1,0\leq j\leqslant d-1\}$ forms a PBW basis for $H_{n,d}$.

In case $d=n$, then $H_n=H_{n,n}$ is the $n^2$-dimensional Taft (Hopf) algebra (see \cite{Ta}). For this reason, $H_{n,d}$ is  called
a generalized Taft algebra in \cite{chyz,Hcp}. The Hopf algebra $H_{n,d}$
can be also approached by quiver and relations, that is, $H_{n,d}$ is isomorphic to the quiver quantum group $KZ_n(q)/I_d$ constructed by Cibils in
\cite{Cib}.

In order to compute the Green ring of $H_{n,d}$, we need to recall the classification of finite dimensional indecomposable representations of $H_{n,d}$.  Thanks to Theorem 4.3 in \cite{chyz}, we have the following structures of $H_{n,d}$ as algebras. Throughtout the paper, we let $m=n/d$. Then
\begin{lemma} Let $H_d$ be the Taft algebra of dimension $d^2$. Then
$$H_{n,d}\cong \stackrel{m\,\, \mathrm{copies}}{\overbrace{H_d\times H_d\times \cdots \times H_d}}.\ \ \ \ \ \  \hfill\Box $$
\end{lemma}

Recall that an algebra A is called Nakayama if each indecomposable
projective left and right module has a unique composition series.
Since $H_{n,d}$  is a Nakayama algebra, its representation theory is
not difficult to describe (see \cite{ARS,Cib, Hcp}). In \cite{Cib}
Cibils  classified the indecomposable modules over $kZ_n(q)/I_d$,
and gave the decomposition formulas of the tensor product of two arbitrary
indecomposable modules. However, using the forms of the presentations in
\cite{Cib}, it is hard to work out explicitly the generators and the
generating relations for the Green rings of $H_{n,d}$. For this reason,
we will rewrite in the following  all indecomposable modules over
$H_{n,d}$, and  reformulate the decomposition formulas of the tensor
product of two  indecomposable modules.

Let ${\mathbb Z}_n:={\mathbb Z}/(n)$. For  $1\leqslant l\leqslant
d$ and $i\in \mathbb{Z}_n,$ denote by $M(l, i)$ the $k$-vector space
with a $K$-basis $\{v_0, v_1, \cdots, v_{l-1}\}.$  Let $\omega$  be
a primitive $n$-th root of unity so that $\omega^m=q.$ Define an
action of $H_{n,d}$ on $M(l, i)$ as follows:\\
 $g\cdot v_j=\omega^iq^{-j}v_j$
for all $0\leqslant j\leqslant l-1$ and
$$h\cdot v_j=\left\{\begin{array}{ll}
v_{j+1},& 0\leqslant j\leqslant l-2,\\
0,& j=l-1.\\
\end{array}\right.
$$
Then $M(l, i)$ is an indecomposable module and the set $\{M(l,i)\ | i\in{\mathbb Z}_n, 1\leqslant l\leqslant d\}$
forms a complete list of non-isomorphic indecomposable $H_{n,d}$-modules. Moreover, $M(l,i)$ is simple if and only if $l=1$; and
$M(l,i)$ is projective indecomposable if and only if $l=d$ (for the proof, see \cite{ Coz,Cib,Hcp}).

\begin{lemma} We have the following isomorphisms:
\begin{eqnarray*}
&& M(1, i) \otimes  M(l, r)\cong M(l, r)\otimes M(1,i)\cong M(l, r+i),\\
 && M(l, r)\cong M(1,r)\otimes M(l,0)\cong M(l,0)\otimes M(1,r)
 \end{eqnarray*}
for all $1\leqslant l\leqslant d$ and $i, r\in{\mathbb Z}_n$.
\hfill$\Box$
\end{lemma}

To find the decomposition of the tensor product $M(l,i)\otimes M(s,
j)$, it is enough to find the decomposition of the tensor product
$M(l,0)\otimes M(s, 0)$. The following proposition tell us that it
suffices to give an explicit evaluation of the tensor product of any
indecomposable representation with the two-dimensional
indecomposable representation $M(2,0).$

\begin{proposition}  We have the following isomorphisms:
\begin{enumerate}
\item  $M(2,0)\otimes M(l, 0)\cong M(l+1,0)\bigoplus M(l-1,-m)$
for all $2\leqslant l\leqslant d-1$ and $d> 2.$
\item  $M(2,0)\otimes M(d,0)\cong M(d,0)\bigoplus M(d,-m)$.
\end{enumerate}
\end{proposition}
\begin{proof}
(1) For $d> 2$ and $2\leqslant l\leqslant d-1$, let $\{v_0, v_1\}$ and $\{w_0,w_1,\cdots, w_{l-1}\}$ be the standard basis of indecomposable representations $M(2,0)$ and $M(l, 0)$ respectively.
Consider the submodule $V$ of $M(2,0)\otimes M(l, 0)$ generated by $v_0\otimes w_0$.
 From  the action $g\cdot h^i\cdot(v_0\otimes w_0)=q^{-i}h^i \cdot(v_0\otimes w_0)$,
 it follows that $V=<h^i\cdot (v_0\otimes w_0)\, |\, i\in \mathbb{N}>.$  By induction, we obtain
$$
h^i\cdot (v_0\otimes w_0)=v_0\otimes w_{i}+(1+q^{-1}+\cdots +q^{-(i-1)})v_1\otimes  w_{i-1}
$$
for $0\leq i\leq l-1,$ and
$$
h^l\cdot (v_0\otimes w_0)=(1+q^{-1}+\cdots +q^{-(l-1)})v_1\otimes w_{l-1}\not=0
$$
since $q$ is a primitive $d$-th  root of unity and $0\leq l\leq d-1$.
It follows that $h^{l+1}\cdot (v_0\otimes w_0)=0.$  By the
construction of $M(l+1,0),$ this implies that $M(2,0)\otimes M(l,
0)$ contains an indecomposable submodule $V$  isomorphic to
$M(l+1, 0).$ On the other hand, if we set
$$e_i=h^i\cdot (v_0\otimes w_1-q^{-(l-1)}(1+q+\cdots +q^{l-2})v_1\otimes w_0) $$
for $i\in \mathbb{N},$ then  by induction again we have
$$
e_i=v_0\otimes w_{i+1}-q^{-(l-1)}(1+q+\cdots +q^{l-i-2})v_1\otimes w_i
$$
for $0\leq i\leq l-2$  and, in particular, we have
$$e_{l-2}=v_0\otimes w_{l-1}-q^{-(l-1)}v_1\otimes w_{l-2}\not=0\  \mathrm{and}\ e_{l-1}=0.$$
Notice that $g\cdot e_i=q^{-i-1}e_i=\omega^{-m}q^{-i}e_i$. Thus,
$M(2,0)\otimes M(l, 0)$ contains an indecomposable submodule $M$
 generated by $e_0$ and isomorphic to $M(l-1, -m)$. Since
$M(l+1, 0)$ has a unique composition series
$$
0\rightarrow M(1,-lm)\rightarrow \cdots \rightarrow M(l-1,
-2m)\rightarrow M(l,-m)\rightarrow M(l+1,0),
$$
we have $M(l-1,-m)\cap M(l+1, 0)=\{0\}$. Comparing the
dimensions we obtain  $M(2,0)\otimes M(l, 0)\cong M(l+1,0)\bigoplus
M(1,-m).$

(2) As in the proof of (1), we let $\{v_0, v_1\}$ and $\{w_0,w_1,\cdots, w_{d-1}\}$ be the standard basis of the indecomposable representations $M(2,0)$ and $M(d, 0)$ respectively. Consider the submodule $V$ of $M(2,0)\otimes M(d, 0)$ generated by $v_0\otimes w_0$. Since $g\cdot h^i\cdot(v_0\otimes w_0)=q^{-i}h^i \cdot(v_0\otimes w_0)$, we have \\
 $V=<h^i\cdot (v_0\otimes w_0)\, |\, i\in \mathbb{Z}>$.
Now by induction, we get
$$
h^i\cdot (v_0\otimes w_0)=v_0\otimes w_{i}+(1+q^{-1}+\cdots
+q^{-(i-1)})v_1\otimes w_{i-1}
$$
for $0\leq i\leq d-1,$ and
$$
h^d\cdot (v_0\otimes w_0)=(1+q^{-1}+\cdots +q^{-(d-1)})v_1\otimes
w_{d-1}=0.
$$
 These imply that $M(2,0)\otimes M(d, 0)$ contains an indecomposable submodule $V$ isomorphic to $M(d, 0).$

 On the other hand, if we let $N$ be  the submodule of $M(2,0)\otimes M(d, 0)$ generated by $v_0\otimes w_1$, then by induction, we obtain the following:
$$
h^i \cdot (v_0\otimes w_1)=v_0\otimes w_{i+1}+(q^{-1}+\cdots
+q^{-i})v_1\otimes w_{i}
$$
for $0\leq i\leq d-2$  and
$$h^{d-1}\cdot (v_0\otimes
w_1)=(q^{-1}+\cdots +q^{-(d-1)})v_1\otimes w_{d-1}\not=0.$$
Thus we get $h^d \cdot(v_0\otimes w_1)=0.$ Note that $g\cdot h^i\cdot
(v_0\otimes w_1)=q^{-i-1}v_0\otimes w_1=\omega^{-m}q^{-i}v_0\otimes
w_1.$  It follows that  $M(2,0)\otimes M(d, 0)$ contains an
indecomposable submodule $N$ isomorphic to $M(d, -m).$
 Comparing the dimensions we obtain that   $M(2,0)\ot M(d, 0)\cong M(d,0)\bigoplus M(d,-m).$
\end{proof}

In order to describe  the Green ring $r(H_{n,d})$ of $H_{n,d}$, we
also need the decomposition of the tensor product $M(d,0)\otimes
M(d,0).$

\begin{proposition} We have the following decomposition:
$$
M(d,0)\otimes M(d,0)\cong M(d, 0)\oplus M(d,-m)\oplus \cdots \oplus M(d, -(d-1)m).
$$
\end{proposition}
\begin{proof}
Let $\{v_0,v_1,\cdots, v_{d-1}\}$ be the standard basis of the
indecomposable representation $M(d,0).$  Similar to the proof of the
Proposition 2.3, for $0\leq i\leq d-1$ we get that
$\langle v_0\otimes v_i\rangle,$ the submodule  of  $M(d,0)\otimes M(d,0)$
generated by element $v_0\otimes v_i$,  is isomorphic to $M(d, -im).$
Now comparing the dimensions, we have an isomorphism:
$$ M(d,0)\otimes M(d,0)\cong M(d, 0)\oplus M(d,-m)\oplus \cdots
\oplus M(d, -(d-1)m).
$$
\end{proof}

\section{\bf Generators and relations for the Green ring $r(H_{n,d})$}

The aim of this section is to describe explicitly the generators and
the generating relations for the Green  ring $r(H_{n,d})$ of the
generalized Taft algebra $H_{n,d}$.

Let $H$ be a Hopf algebra over a field $K$. Recall that the Green ring or the
representation ring  $r(H)$ of $H$ can be defined as
follows. As a group $r(H)$ is the free abelian group generated by
the isomorphism classes $[V]$ of finite dimensional $H$-modules $V$
modulo the relations $[M\oplus V]=[M]+[V]$. The multiplication of
$r(H)$ is given by the tensor product of $H$-modules, that is,
$[M][V]=[M\ot V]$. Then $r(H)$ is an associative ring with identity given by
$[K_\varepsilon]$, where $K_\varepsilon$ is the trivial
1-dimensional $H$-module. Note that $r(H)$ is a free abelian group
with a $\mathbb Z$-basis $\{[V]|V\in{\rm ind}(H)\}$, where ${\rm
ind}(H)$ denotes the set of finite dimensional indecomposable
$H$-modules. The complexified Green ring $R(H)$, which we shall call
the Green algebra for short, is an associative
$\mathbb{C}$-algebra defined by $\mathbb{C}\ot_{\mathbb Z}r(H)$.

Now we return to the case  $H=H_{n,d}$.  We begin with the
following lemma, which follows from Lemma 2.2, Proposition 2.3 and 2.4.

\begin{lemma} The following equations hold in $r(H_{n,d})$.
\begin{enumerate}
\item $[M(1,-1)]^n=1$ and $[M(l, r)]=[M(1,-1)]^{n-r}[M(l,0)]$ for all $1\leqslant l\leqslant d$ and $r\in{\mathbb Z}_n$.
\item $[M(l+1, 0)]=[M(2,0)][M(l,0)]-[M(1,-1)]^m[M(l-1,0)]$ for all $2\leqslant l\leqslant d-1$ and $d>2$.
\item $[M(2,0)][M(d,0)]=(1+[M(1,-1)]^m)[M(d,0)]$.
\item  $[M(d,0)][M(d, 0)]=\sum\limits_{i=0}^{d-1}[M(d, -im)].$ \hfill$\Box$
\end{enumerate}
\end{lemma}

As a consequence, we have the following corollary.

\begin{corollary}
The Green ring $r(H_{n,d})$  is a commutative ring generated by $[M(1,-1)]$ and $[M(2,0)].$ \hfill$\Box$
\end{corollary}
Therefore, $r(H_{n,d})$ can be identified with a quotient of the
polynomial ring $\mathbb{Z}[y,z].$ In the sequel, we shall determine the relations of the two generators. Let $F_s(y,z)$ be the generalized
Fibonacci polynomials defined by
$$
F_{s+2}(y,z)=zF_{s+1}(y,z)-yF_{s}(y,z)
$$
for $s\geq 1$, while $F_0(y,z)=0, \, F_1(y,z)=1,\, F_2(y,z)=z.$
These generalized Fibonacci polynomials were found appearing in the
generating relations of the Green rings of the Taft algebras $H_n$
(see \cite{Coz}). The general form of the polynomials is as follows:

\begin{lemma} \cite[Lemma 3.11]{Coz}
For $s\geq 2$ we have
\begin{equation*}  F_s(y,z)=\sum\limits_{i=0}^{[(s-1)/2]}(-1)^i\left[
                                              \begin{array}{c}
                                                s-1-i\\
                                                i\\
                                              \end{array}
                                            \right]y^{i}z^{s-1-2i},
\end{equation*}
where $[\frac{s-1}{2}]$ denotes the biggest integer which is not bigger than $\frac{s-1}{2}$. \hfill$\Box$
\end{lemma}

Now we are ready to present the Green ring $r(H_{n,d}).$

\begin{theorem}
The Green ring $r(H_{n,d})$ of $H_{n,d}$ is isomorphic to the
quotient ring of the polynomial ring $\mathbb{Z}[y,z]$ modulo the
ideal $I$ generated by the following elements
 $$
 y^n-1,\,\, (z-y^m-1)F_d(y^m,z).
 $$
 Moreover, under the isomorphism the polynomial $F_s(y^m,z)$ corresponds to $[M(s,0)]$ for $2\leq s\leq d.$
\end{theorem}
\begin{proof}
It follows  from Lemma 3.1 that
$$
([M(2,0)]-(1+[M(1,-1)]^m))[M(d,0)]=0,\,   \, [M(1,-1)]^n=1
$$
and
$$
[M(d,0)]=F_d([M(1,-1)]^m, [M(2,0)]).
$$
These lead to a natural surjective $\mathbb{Z}$-algebra homomorphism
$$\Phi: \mathbb{Z}[y,z]/I\rightarrow r(H_{n,d}),\,
\Phi(y)=[M(1,-1)],\, \Phi(z)=[M(2,0)]
$$
from $\mathbb{Z}[y,z]/I$ to $r(H_{n,d})$.
Observe that as a free $\mathbb{Z}$-module, $\mathbb{Z}[y,z]/I$ has a
$\mathbb{Z}$-basis $\{y^iz^j\, |\, 0\leq i\leq n-1, 0\leq j\leq
d\}.$  This means that $\mathbb{Z}[y,z]/I$ and  $r(H_{n,d})$  both have
the same rank $nd$ as free $\mathbb{Z}$-modules. It follows that
$\Phi$ must be an isomorphism. The other assertions now follow from the
fact that  the following relations hold in $r(H_{n,d})$
$$
[M(l+1, 0)]=[M(2,0)][M(l,0)]-[M(1,-1)]^m[M(l-1,0)]
$$
for all $2\leqslant l\leqslant d-1$ and $d>2$.
\end{proof}

 One can easily see that the Grothendieck ring of $H_{n,d}$ is the group ring $\mathbb{Z}\mathbb{Z}_n$ generated by $[M(1,-1)]$.
 From the above theorem, we see that the Green ring of $H_{n,d}$ is much more complicated than the Grothendieck ring of $H_{n,d}$. Another interesting factor is that the Green ring of $H_{n,d}$ is commutative although $H_{n,d}$ is not quasitriangular. Thus the Green rings of the generalized Taft algebras
$H_{n,d}$ (here $n,d$ can be chosen) provide more examples that  a non-quasitriangular Hopf algebra possesses a commutative Green ring. However, we do have non-commutative Green rings. For example, the Green ring of the small quantum group is not commutative when $p\geq 3$ (see \cite{ks}).

Now let $A$ be the subalgebra of $r(H_{n,d})$ generated by $[M(1,-m)]$
and $[M(2,0)]$. Then by Lemma 3.1 we have
$$
r(H_{n,d})=A\bigoplus A[M(1,-1)]\bigoplus\cdots \bigoplus
A[M(1,-1)]^{(m-1)}. $$ By Theorem 3.10 in \cite{Coz} and Theorem
3.4, it is easy to see that $A$ is isomorphic to $r(H_d),$ the Green
ring of the Taft algebra $H_d.$ Now the following theorem is straightforward.

\begin{theorem}
As a $\mathbb{Z}$-algebra, the  Green ring $r(H_{n,d})$ is
isomorphic to $r(H_d)[x]/(x^m-a),$ where $a=[M(1,-m)].$ \hfill$\Box$
\end{theorem}

\begin{remark}
For a finite-dimensional Hopf algebra $H$ over a field $K$, the antipode
$S$ induces an anti-ring endomorphism $\ast$ of the Green ring
$r(H)$, see \cite{Lo}. In the case where $H$ is the generalized Taft algebra
$H_{n,d},$  it is easy to see that the anti-ring endomorphism $\ast$
is given by $[M(1,-1)]^\ast=[M(1, -1)]^{-1}$ and
$[M(2,0)]^\ast=[M(1, -1)]^{-m}[M(2, 0)].$ It follows that the
the anti-isomorphism $\ast$ is an involution whereas $S$ is not.
This is because $S^2$ is inner and $V^{**}\cong V$ for any finite dimensional $H$-module, see \cite{Lo}.
\end{remark}

\section{\bf Roots of the generalized Fibonacci polynomials}\selabel{3}

In order to determine all nilpotent elements and the structure of
the Green rings $r(H_{n,d})$ and $R(H_{n,d})$, we need to compute
explicitly all roots of the generalized Fibonacci polynomial $F_d(y,z)$
associated to the generating relations of the Green ring
$r(H_{n,d}).$  We shall set up in this section a general framework, which is of interest in its own.

Let $\mathbb{C}[x]$ be the polynomial algebra over the complex number field $\mathbb{C}$ with variable $x$.
Fix $0\not=a\in \mathbb{C}$, the generalized Fibonacci polynomials  $F_s(a, x)$ are defined recursively:
$$
F_0(a,x)=0,\,  \,  F_1(a, x)=1,\, \,
F_{s+2}(a,x)=xF_{s+1}(a,x)-aF_s(a,x)
$$
for $s\geq 0$.  $F_s(a, x)$ has the following combinatorial formula as
shown in Section 3:
$$
F_{s}(a,x)=\sum\limits_{i=0}^{[(s-1)/2]}(-1)^i\left[
                                              \begin{array}{c}
                                                s-1-i\\
                                                i\\
                                              \end{array}
                                            \right]a^ix^{s-1-2i}
 $$
 for $s\geq 1.$ From these combinatorial expressions, it is easy to see that $F_{s}(a,x)$ is a polynomial in $\mathbb{C}[x]$ with degree $s-1$.

\begin{remark} The classical Fibonacci polynomials $F_s(x)$ are defined by the linear recurrence
$$
F_{s+2}(x)=xF_{s+1}(x)+F_{s}(x)
$$
for $s\geq 0$ and $F_0(x)=0, F_1(x)=1,$  that is,
$F_s(x)=F_s(-1,x).$ These polynomials are of great importance  in
the study of many subjects such as algebra, geometry, and number
theory. Obviously,  these polynomials have a close
relation with the famous Fibonacci numbers. The Fibonacci
polynomials are essentially Chebychev polynomials, of which the roots can be computed. Thus $F_s(x)$  has $s-1$ distinct roots given by $($see, e.g., \cite{HB}$):$
$$
x_j=2i\, cos \frac{\pi j}{s}, \,\,   1\leq j\leq s-1.\ \ \ \ \ \ \ \ \ \  \hfill \Box
$$
\end{remark}

The following proposition give us the roots of generalized Fibonacci polynomial $F_s(a,x)$.
\begin{proposition}
Let $s\geq 2.$  Then the generalized Fibonacci polynomial $F_s(a,x)$ has $s-1$ distinct complex  roots given by
$$
x_j=2\sqrt{a}\; cos \frac{j\pi}{s},  \, \, \, \, \, 1\leq j\leq s-1.
$$
\end{proposition}
\begin{proof}
Let $x=by$ and $b=-i\sqrt{a}\in \mathbb{C}.$  Then
$$
F_s(a, x)=b^{s-1}F_s(-1,y)=b^{s-1}F_s(y).
$$
Note that the roots of $F_s(y)$ are $y_j=2i\, cos \frac{j\pi}{s}, \, \,
\, \, \, 1\leq j\leq s-1.$ It follows that the roots of $F_s(a,x)$
are
$$
x_j=2\sqrt{a}\; cos \frac{j\pi}{s},  \, \, \, \, \, 1\leq j\leq s-1.
$$

\end{proof}

We need the following corollary  in Section 5.
\begin{corollary}
For every $1\leq j\leq s-1$, there exists a $2s$-th root $\eta_j\not=1$
of unity, such that
$$
x_j=2\sqrt{a}\; cos\frac{\pi j}{s}=\sqrt{a}({\eta_j}+\eta_j^{-1}).
$$
\end{corollary}
\begin{proof}
For each $j$, $1\leq j\leq s-1,$ let $\eta_j=cos \frac{j\pi }{s}+isin \frac{j\pi }{s}=e^{\frac{j\pi i }{s}}.$ Then  $\eta_j\not=1$ is a $2s$-th root of unity and
  $$
 2\sqrt{a}\; cos\frac{\pi j}{s}=\sqrt{a}({\eta_j}+\eta_j^{-1}).
 $$
 \end{proof}

To end this section, we  solve the system of equations determined by
the generating relations of the Green ring $r(H_{n,d}).$ For $0\leq
k\leq n-1,$ let $\omega_k=cos\frac{2k\pi }{n}+isin\frac{2k\pi }{n}$
be an $n$-th root of unity. By Proposition 4.2, all distinct roots
of the generalized Fibonacci polynomial $F_d(\omega_k^m,x)$ are
given by
$$\sigma_{k,j}=2\sqrt{\omega_k^m}\; cos\frac{\pi j}{d}, \,\,\,  1\leq j\leq d-1. $$

We have the following lemma.

\begin{lemma}
 The following system of equations
\begin{equation}\label{system}
 \begin{cases}
 \,\, \, y^n=1,& \\
\, \,  (z-y^m-1)F_d(y^m,z)=0 &\\
\end{cases}
\end{equation}

 has $(d-1)n+m$ distinct solutions in $\mathbb{C}$ , where $m=n/d$, and the solutions are given by
 \begin{equation}\label{solutions}
 \begin{array}{rl}
 \mathfrak{T}=&\{ (\omega_k,2)\, |\, 0\leq k\leq n-1, d|k\}\\
 & \cup\{(\omega_k, \sigma_{k,j}) \, |\, 0\leq k\leq n-1,1\leq j\leq d-1\}.
 \end{array}
 \end{equation}
\end{lemma}

\begin{proof}
By Proposition 4.2, it is not difficult to see that the solutions of
the system (\ref{system}) are as follows:
$$
 \{ (\omega_k, 1+\omega_k^m), (\omega_k, \sigma_{k,j}) \, |\, 0\leq k\leq n-1, 1\leq j\leq d-1\}.
 $$
Thanks to Corollary 4.3, the roots of the generalized Fibonacci
polynomial $F_d(\omega_k^m, x)$ can be written as
$$
\sigma_{k,j}=2\sqrt{{\omega_k}^m}\; cos\frac{\pi
j}{d}=\sqrt{\omega_k^m}({\eta_j}+\eta_j^{-1}),
$$
where $\eta_{j}=cos\frac{j\pi }{d}+isin\frac{j\pi }{d}$ is a $2d$-th root of unity,  and  $1\leq j\leq d-1$. Now we divide $k$, $0\leq k\leq n-1$, in two cases. \\
Case 1:  $d\mid k$.  In this case, for $1\leq j\leq d-1$ we have $\sigma_{k,j}\not=2$.   Thus, the following set gives  $md$ distinct roots of the system (\ref{system}):
 $$
 \{ (\omega_k,2), (\omega_k, \sigma_{k,j}) \, |\, 0\leq k\leq n-1, d|k, 1\leq j\leq d-1\}.
 $$
Case 2: $d\nmid k$.  In this case,  $\eta_{k}=\sqrt{\omega_k^m}\not=1$. So we have
$$
\sigma_{k,k}=\sqrt{\omega_k^m}({\eta_k}+\eta_{k}^{-1})=
\sqrt{\omega_k^m}(\sqrt{\omega_k^m}+\sqrt{\omega_k^{-m}})=1+\omega_k^m.
$$
Since the generalized Fibonacci polynomial $F_d(\omega_k^m, z)$ possesses
$d-1$ distinct roots $\sigma_{k,j}$ for each $k,$ this implies that
for each $\omega_{k}^m\not=1$, the system (\ref{system})
%$$ \begin{cases}
% \,\, \, y^n=1,& \\
%\, \,  (z-y^m-1)F_d(y^m,z)=0 &\\
%\end{cases}
%$$
has $d-1$ distinct roots $\{(\omega_k, \sigma_{k,j})\, |\, 1\leq j\leq d-1\}$.\\
To summarize, we have in total
$$md+(n-m)(d-1)=nd-n+m$$
distinct solutions  for the system (\ref{system})  of  equations.
\end{proof}

\section{\bf Nilpotent elements in $r(H_{n,d})$ and representation theory of $R(H_{n,d})$ }

In this section, we determine all nilpotent elements in
the Green ring $r(H_{n,d})$ and  classify  the  finite
dimensional indecomposable representations over the Green algebra
$R(H_{n,d})$. We first list the irreducible representations of the Green algebra $R(H_{n,d}).$

\begin{theorem}
Let $d\geq 2.$ Then the Green algebra $R(H_{n,d})$ has exactly $nd-n+m$
irreducible modules and each irreducible module is 1-dimensional.
\end{theorem}

\begin{proof}
By Theorem 3.4, we know that the Green algebra $R(H_{n,d})$ of $H_{n,d}$
is commutative and isomorphic to $\mathbb{C}[y,z]/I$, where the
ideal $I$ is generated by the following elements
$$
y^n-1,\,\,\, (z-y^m-1)F_d(y^m,z).
$$
Since $\mathbb{C}$ is an algebraically closed field of characteristic 0, the dimension of each irreducible module is 1. By Lemma 4.4,  we know that the system (\ref{system})
%$$ \begin{cases}
% y^n=1,& \\
% (z-y^m-1)F_d(y^m,z)=0 &\\
%\end{cases}
%$$
 has $nd-n+m$ distinct solutions given by $\mathfrak{T}$ in (\ref{solutions}).
%\begin{align*}
% \mathfrak{T}=&\{ (\omega_k,2)\, |\, 0\leq k\leq n-1, d|k\}\\
% & \cup\{(\omega_k, \sigma_{k,j}) \, |\, 0\leq k\leq n-1,1\leq j\leq d-1\}.
% \end{align*}
For each solution $(\lambda,\mu)\in \mathfrak{T},$ one can define an
irreducible $R(H_{n,d})$-module $\mathbb{C}_{\lambda,\mu}$ on the vector space $\mathbb{C}$  by $y\cdot 1=\lambda, z\cdot 1=\mu.$

It is clear that $(\lambda,\mu)\mapsto \mathbb{C}_{\lambda,\mu}$ gives a one to one correspondence between the set of solutions for the system (\ref{system}) and the set of the isomorphism classes of irreducible modules over $R(H_{n,d})$.
\end{proof}

 \begin{theorem}
Let $d\geq 2.$  The set of nilpotent elements in $r(H_{n,d})$ is equal to
$$
\langle [M(d,i)]-[M(d,j)]\, |\, i\equiv j (mod\,\; m)\rangle,
$$
that is, the Jacobson radical $J(r(H_{n,d}))$ of $r(H_{n,d})$ has a
$\mathbb{Z}$-basis
$$
\{ [M(d,im+j)]-[M(d,(i-1)m+j)]\,|\, 1\leq i\leq  d-1, 0\leq j\leq m-1\}.$$
\end{theorem}
\begin{proof}
By Proposition 2.4 we  have the $H_{n,d}$-modules decomposition
$$
M(d,0)\ot M(d, 0)\cong \bigoplus\limits_{i=0}^{d-1}M(d, -im).$$
Thus, for $i,j\in \mathbb{Z}_n$, we get
\begin{align*}
M(d,i)\ot M(d, j)&\cong M(1,i)\ot M(1,j)\ot M(d, 0)\ot M(d, 0)\\
&\cong M(1,i+j)\ot \bigoplus\limits_{s=0}^{d-1}M(d, -sm)\\
&\cong \bigoplus\limits_{s=0}^{d-1}M(d, i+j-sm).
\end{align*}
This implies that if $i\equiv j (mod\ m)$, then we have
$$
M(d,i)\ot M(d, i)\cong M(d,j)\ot M(d, j)\cong
\bigoplus\limits_{s=0}^{d-1}M(d, i+j -sm)\cong M(d,i)\ot M(d, j).
$$
It follows that in the Green ring $r(H_{n,d})$
$$
[M(d,i)]^2=[M(d, j)]^2=[M(d,i)][M(d, j)]=[M(d, j)][M(d,i)]
$$
for $i\equiv j (mod\ m)$. Therefore, if $i\equiv j (mod\ m)$, we
have that $$([M(d, i)]-[M(d,j)])^2=0,$$
i.e., $0\not=[M(d,
i)]-[M(d,j)]$ is an nilpotent element in $r(H_{n,d})$. Since
$\mathbb{Z}$ is a principal ideal integral domain and hence each
submodule of free $\mathbb{Z}$-module is free. Note that the set
$$ \{ [M(d,im+j)]-[M(d,(i-1)m+j)]\,|\, 1\leq i \leq  d-1, 0\leq
j\leq m-1\}\subset J(r(H_{n,d}))$$ is independent over $\mathbb{Z}.$
Hence the rank $r$ of $J(r(H_{n,d}))$ as $\mathbb{Z}$-module  is
large or equal to $(d-1)m=n-m.$ On the other hand, it is easy to see
from Theorem 5.1 that the dimension
$dim_{\mathbb{C}}(J(R(H_n)))=n-m.$ But the fact that
$dim_{\mathbb{C}}(J(R(H_n)))\geq r$ is clear, and hence $r=n-m.$
Thus, the set
$$
\{ [M(d,im+j)]-[M(d,(i-1)m+j)]\,|\, 1\leq i\leq  d-1, 0\leq j\leq m-1\}
 $$ forms a $\mathbb{Z}$-basis of $J(r(H_{n,d}))$.
\end{proof}
As a consequence, we obtain the following:
\begin{corollary}
The Jacobson radical $J(r(H_{n,d}))$ is a principal ideal generated
by the element $([M(1,m)]-1)[M(d,0)].$
\end{corollary}
\begin{proof} For $1\leq i \leq d-1, 0\leq j\leq m-1,$ by Lemma 3.1, we have
\begin{align*}
&[M(d,im+j)]-[M(d,(i-1)m+j)]\\
=& ([M(1,im+j)]-[M(1,(i-1)m+j)])[M(d,0)]\\
=&[M(1,(i-1)m+j)]([M(1,m)-1])[M(d,0)].
\end{align*}
Note that $[M(1,(i-1)m+j)]$ is invertible in $r(H_{n,d})$. Hence, the assertion follows from Theorem 5.2.
\end{proof}

In the sequel, we shall classify the indecomposable but non-irreducible finite dimensional modules over the Green algebra $R(H_{n,d})$. For each $k,$ $0\leq k\leq n-1$, such that $d\nmid k,$  write
$\omega_k$ for $cos\frac{2k\pi }{n}+isin\frac{2k\pi }{n}$, an $n$-th root
of unity, and let $V(k)$ be a 2-dimensional $\mathbb{C}$-vector space
with a basis $\{v_1, v_2\}.$ Define an action of the Green algebra
$R(H_{n,d})$ on $V(k)$ as follows:
$$
y\cdot v_i=\omega_k v_i,\, z\cdot v_1=(1+\omega_k^m) v_1,  \,  z\cdot v_2=v_1+(1+\omega_k^m) v_2.
$$
Then we have the following.
\begin{lemma}
 Let $0\leq k,t\leq n-1, d\nmid k, d\nmid t.$  Then $V(k)$ is an   indecomposable and reducible module of $R(H_{n,d})$. Moreover,
 $V(k)\cong V(t)$ if and only if $k=t.$
\end{lemma}
\begin{proof}
 $V(k)$ is indecomposable because $1+\omega_k^m$ is a double root of the
equation
$$(z-\omega_k^m-1)F_d(\omega_k^m, z)=0
$$
by Lemma 4.4 and the action of $z$ corresponds to the following Jordan block
$$\left(
 \begin{array}{cc}
 1+\omega_k^m  & 1  \\
 0  & 1+\omega_k^m \\
 \end{array}\right).
 $$
Observe that $\mathbb{C}v_1$ is a submodule of $V(k)$. Therefore,
$V(k)$ is reducible which is direct from Theorem 5.1. The rest is
straightforward.
\end{proof}

\begin{theorem}
Let $V$ be a finite dimensional indecomposable and reducible module of the  Green algebra $R(H_{n,d}).$ Then there exists $k$,  $0\leq k\leq n-1$, and $d\nmid k$, such that $V\cong V(k).$
\end{theorem}

\begin{proof}
Assume that $V$ is a finite dimensional indecomposable module with
dimension $s\geq 2.$ Since the actions of $y$ and $z$ on $V$
are commutative and $y^n=1,$  there exist a basis
$\{v_1, v_2,\cdots, v_s\}$ of $V$ such that with respect to this basis, the matrices  corresponding to the actions  of the generators $y$  and $z$ are respectively $Y=\omega_k E_s$, where
$\omega_k=cos\frac{2k\pi }{n}+isin\frac{2k\pi }{n}$ is an $n-$th
root of unity for some $0\leq k\leq n-1$ and $E_s$ is the $s\times s$ identity  matrix, and the Jordan block:
$$Z=\left(
                           \begin{array}{ccccc}
                             \lambda & 1       &   &  & \\
                              &      \lambda & 1 &  &\\
                              &  & \ddots & \ddots  &\\
                              &  &  & \lambda & 1 \\
                              &   &  &  &   \lambda\\
                           \end{array}
                         \right)_{s\times s}
$$
for some $\lambda\in \mathbb{C}$. Notice that the matrix $Z$ must satisfy the following matrix equation
$$
(Z-(\omega_k^m+1)E_s)\prod\limits_{j=1}^{d-1}(Z-\sigma_{k,j}E_s)=0,
$$
where $\{\sigma_{k,j}\,| \, 1\leq j\leq d-1\}$ is the set of roots
of the generalized Fibonacci polynomial $F_d(\omega_k^m, z),$ see
Proposition 4.2.

Since for each $j$ the matrix $Z-\sigma_{k,j}E_s$ is a Jordan block,
$Z-\sigma_{k,j}E_s$ is invertible if $\lambda\not=\sigma_{k,j}$ .
Moreover, $\sigma_{k,j}$ for $1\leq j\leq d-1$ are distinct and
$\sigma_{k,k}=1+\omega_k^m$ for $d\nmid k$. It follows that
$\omega_k^m\not=1, s=2$  and $\lambda=1+\omega_k^m.$ Therefore, we
have that $V\cong V(k)$ for some $0\leq k\leq n-1$ such that $d\nmid
k.$
\end{proof}

Now we can describe the blocks of the Green algebra $R(H_{n,d}).$
\begin{corollary}
\begin{enumerate}
\item The  set $\{\mathbb{C}_{\lambda,\mu}, V(k) \, | (\lambda,
\mu)\in \mathfrak{T}, 0\leq k\leq n-1, d\nmid k\}$ forms a complete
list of finite dimensional indecomposable representations of
$R(H_{n,d})$ with cardinal number $nd.$ Moreover,
$\mathbb{C}_{\lambda,\mu}$ is projective if and only if
$\mu\not=1+\lambda^m$ and $V(k)$ is projective for each $k.$
 \item There are $nd-2(n-m)$ blocks of dimension $1$ and $n-m$ blocks of dimension $2$. \hfill$\Box$
 \end{enumerate}
\end{corollary}

\section{\bf Comparing the Green ring and the projective class ring of $H_{n,d}$ }

In this section, we compare the structure of the Green ring $r(H_{n,d})$
and the structure of the projective class ring $p(H_{n,d}).$  Recall
that the projective class ring $p(H_{n,d})$ of  $H_{n,d}$ is the subalgebra of $r(H_{n,d})$ generated by the
projective and semisimple representations of $H_{n,d}$ (see
\cite{CC} for details).  It is easy to see that $p(H_{n,d})$ has a
$\mathbb{Z}$-basis $\{[M(1,i)], [M(d, i)]\,|\, 0\leq i\leq n-1\}.$
In \cite{CC} Cibils determined the structure of the complexified
projective class algebra $P(H_d)$ of the Tafe algebra $H_d$, which is
isomorphic to $\mathbb{C}^2\times {\mathbb{C}[\epsilon]}^{d-1}$,
where $\mathbb{C}[\epsilon]$ is the algebra of dual numbers
$\mathbb{C}[x]/(x^2)$. Applying the same technique
used on $r(H_{n,d})$ and $R(H_{n,d})$,  we can easily obtain the
following result.

\begin{proposition} Let $p(H_{n,d})$ be the projective class ring of $H_{n,d}$ and assume $d\geq 2.$ Then
 as a $\mathbb{Z}$-algebra, $p(H_{n,d})$ is generated by $[M(1,-1)]$ and $[M(d, 0)]$, and is isomorphic to $\mathbb{Z}[y,z]/I$  with  $I$ generated by the following elements:
$$
y^n-1,\,\,   z^2-(1+y^m+y^{2m}+\cdots+y^{(d-1)m})z.
$$
\end{proposition}

\begin{proof}
Follows from Lemma 3.1.
\end{proof}

Since the nilpotent elements of $r(H_{n,d})$ stem from projective indecomposable modules, we have that the Jacobson radical of $r(H_{n,d})$ is equal to the Jacobson radical of $p(H_{n,d})$.

\begin{corollary} The set of nilpotent elements of $p(H_{n,d})$ is equal to the set of nilpotent elements of $r(H_{n,d}).$ \hfill$\Box$
\end{corollary}

The following proposition summarizes the irreducible representations of the complexified projective class algebra
$P(H_{n,d})=\mathbb{C}\ot _{\mathbb{Z}} p(H_{n,d})$.

For each $k$, $0\leq k\leq n-1$,  let $\omega_k=cos\frac{2k\pi }{n}+isin\frac{2k\pi }{n}$ be an $n$-th root of unity. Let $\mathbb{C}_k=\mathbb{C}$ for each $k$, and $\mathbb{C}_{k,d}=\mathbb{C}$ for $d\mid k$, as vector spaces.  Define the actions of $P(H_{n,d})$
on $\mathbb{C}_k$ and $\mathbb{C}_{k,d}$ respectively as follows:
$$
y\cdot 1=\omega_k,\, z\cdot 1=0
$$
and
$$
y\cdot 1=\omega_k,\, z\cdot 1=d
$$
Then we have the following.

\begin{proposition}
\begin{enumerate}
\item There are $n+m$ non-isomorphic irreducible 1-dimensional $P(H_{n,d})$-representations
$$
\{\mathbb{C}_k \, |\, 0\leq k\leq n-1\}\cup\{\mathbb{C}_{k, d}\,\,
|\,  0\leq k\leq n-1, d\mid k\}
$$
\item  There are $n-m$ non-isomorphic reducible indecomposable 2-dimensional $P(H_{n,d})$-representations $V_k$, where
 $0\leq k\leq n-1, d\nmid k$,  and $V_k$ has a basis $\{v_1, v_2\}$
with the module structure given by
$$
y\cdot v_1= \omega_k v_1, y\cdot v_2=\omega_kv_2,\, z\cdot v_1=0, z\cdot v_2=v_1.
$$
\end{enumerate}
Moreover, the set $\{ \mathbb{C}_k, \mathbb{C}_{k,d}, V_i\}$ forms a complete list of non-isomorphic indecomposable $P(H_{n, d})$-representations.
\hfill$\Box$
\end{proposition}

We finish the paper with a remark on the stable Green ring of
$H_{n,d}$. The stable Green ring was introduced in the study of
Green rings for the modular representation theory of finite groups.
It is a quotient of the Green ring modulo all projective
representations (see\cite{B,EG}).  In our situation, it is easy to
see from Theorem 3.4 that the stable Green ring $St(H_{n,d})$ of
$H_{n,d}$ is generated by $[M(1,-1)]$ and $[M(2,0)]$ and is
isomorphic to $\mathbb{Z}[y,z]/J$, where the ideal $J$ is generated
by the following elements:
$$
y^n-1,\, \,   F_d(y^m,z).
$$
By Theorem 5.2,  the stable Green ring $St(H_{n,d})$ has no
nilpotent elements, and hence it is semiprimitive.

\section*{ACKNOWLEDGMENTS}
\hskip\parindent
The first author would like to thank the Department of Mathematics, University of Hasselt and Stuttgart University
for their hospitality during his visit in 2012. He is grateful to the Belgium FWO for its financial support.
The research was also partially supported by NSF of China (No. 11171291) and Doctorate Foundation of Ministry of Education of China
(No. 200811170001).\\

\end{document}